\documentclass[12pt]{article}		
\usepackage[english]{babel} 		
\usepackage[utf8]{inputenc}		
\usepackage[dvipsnames]{xcolor}
\usepackage{amsmath, amssymb, amsthm}
\usepackage{graphicx}				
\usepackage{hyperref}	
\usepackage{todonotes}	
\usepackage[export]{adjustbox}
\usepackage{subcaption}
\usepackage{tabularx}
\usepackage{enumerate}
\newtheorem{thm}{Theorem}
\newtheorem{lem}{Lemma}
\newtheorem{cor}{Corollary}

\newtheorem{prop}{Proposition}
\newtheorem{defn}{Definition}

\renewcommand\det{\operatorname{Det}}
\newcommand\dist{\operatorname{Dist}}
\newcommand\px{\operatorname{PX}}

\newcommand\aut{\operatorname{Aut}}

\newcommand\inv{^{-1}}   
\newcommand{\Z}{\mathbb Z}

\newcommand{\F}{\mathcal F}

\title{Symmetry Parameters of Praeger-Xu Graphs}
\author{Sally Cockburn \and Max Klivans {\thanks{This research was supported by The Monica Odening '05 Student Internship and Research Fund in Mathematics.}}}
\date{\today}
\begin{document}

\maketitle

\begin{abstract}
    Praeger-Xu graphs are connected, symmetric, 4-regular graphs that are unusual both in that their automorphism groups are large, and in that vertex stabilizer subgroups are also large. Determining number and distinguishing number are parameters that measure the symmetry of a graph by investigating additional  conditions that can be imposed on a graph to eliminate its nontrivial automorphisms. In this paper, we compute the values of these parameters for Praeger-Xu graphs. Most Praeger-Xu graphs are 2-distinguishable; for these graphs we also proved the cost of 2-distinguishing.
\end{abstract}

{\bf Keywords}: Praeger-Xu graphs; determining number; distinguishing number; cost of 2-distinguishing.

{\bf Subject Classification:} 05C15, 05C25, 05C69

\section{Introduction}

A finite simple graph $G = (V, E)$ consists of a finite, nonempty set $V$ of  vertices, and a set of $2$-subsets of $V$, called edges. An appealing feature of graphs is that they can be represented geometrically, with dots corresponding to vertices and lines between dots corresponding to edges. Certain graphs have the property that when the positions of the dots are carefully chosen, this geometric representation displays visual symmetry. However, a graph is an set-theoretic object, not a drawing. There are  various ways to give a rigorous mathematical characterization of symmetry.

 An {\it automorphism} $\alpha$ of a graph $G = (V,E)$ is a permutation of $V$ such that for all $u, v \in V$, $\{u, v\} \in E$ if and only if $\{\alpha(u), \alpha(v)\} \in E$. The set of automorphisms of $G$, denoted $\aut(G)$, is a group under composition. For example, the automorphism group of the complete graph on $n$ vertices is the entire group of permutations on $n$ elements; that is, $\aut(K_n) = S_n$. For $n\ge 3$, the automorphism group of the cycle $C_n$ is the dihedral group $D_n$, consisting of rotations and reflections. 

 One way of characterizing the symmetry of a graph is to determine whether the vertices and/or edges play the same role, in the following sense. A graph $G$ is {\it vertex-transitive} if for all $u, v \in V$, there is $\alpha \in \aut(G)$ such that $\alpha(v) = u$. Similarly, $G$ is {\it edge-transitive} if for all $\{u, v\}, \{x, y\} \in E$, there is $\alpha \in \aut(G)$ such that $\alpha(\{u, v\}) = \{\alpha(u), \alpha(v)\} = \{x, y\}$. More stringently, a graph is {\it arc-transitive} if for all $\{u, v\}, \{x, y\} \in E$, there is $\alpha \in \aut(G)$ such that $\alpha(u) = x$ and $\alpha(v) = y$. Connected, arc-transitive graphs are automatically both vertex-transitive and edge-transitive, and are simply called {\it symmetric} graphs. Complete graphs $K_n$ and cycles $C_n$ are examples of symmetric graphs.

 Another way of characterizing the symmetry of a graph $G$ is to quantify extra measures that can taken to prevent the existence of nontrivial automorphisms of $G$. As one example, we could require that automorphisms  of $G$ fix point-wise a subset $S$ of vertices. If the only automorphism doing so is the identity, then $S$ is called a {\it determining set} of $G$. The {\it determining number} of $G$, denoted $\det(G)$, is the minimum size of a determining set of $G$. (Some authors use the term {\it fixing} instead of determining, for both sets and numbers.) Graphs with no nontrivial automorphisms, sometimes called asymmetric or rigid graphs, have determining number $0$; at the opposite end of the spectrum, $\det(K_n) = n-1$.  A minimum determining set for $C_n$  is any set of two non-antipodal vertices, so $\det(C_n) = 2$.

 As another example, we could paint the vertices with different colors and require that automorphisms preserve set-wise the color classes. Graph $G$  is {\it d-distinguishable} if the vertices can be colored with $d$ colors in such a way that the only automorphism preserving the color classes is the identity. The {\it distinguishing number} of $G$, denoted $\dist(G)$, is the minimum number of colors required for a distinguishing coloring. For a discussion of  elementary properties of determining numbers and distinguishing numbers, as well as the connections between them, see~\cite{AlBo2007}. 
 
 Remarkably, many infinite families of symmetric graphs have been found to have distinguishing number $2$, including hypercubes ~\cite{BoCo2004},
 Cartesian powers $G^{\Box n}$ of a connected graph where $G\ne K_2,K_3$ and $n\geq 2$~\cite{Al2005, ImKl2006,KlZh2007} and Kneser graphs $K_{n:k}$ with $n\geq 6, k\geq 2$~\cite{AlBo2007}. 
 Boutin~\cite{B2008} introduced an additional invariant in such cases; the {\it cost of 2-distinguishing} $G$, denoted $\rho(G)$, is the minimum size of a color class in a 2-distinguishing coloring of $G$.  

In this paper, we find 
these symmetry parameters for a family of symmetric graphs called Praeger-Xu graphs. 
They are remarkable among all connected, symmetric, 4-regular graphs for having very large automorphism groups. Moroever, there is an infinite family of Praeger-Xu graphs with the property that the smallest subgroup of automorphisms that acts transitively on the vertices has an arbitrarily large vertex stabilizer. For these and more results on Praeger-Xu graphs, see \cite{OTCOPXG}, \cite{PSV2010} and \cite{JPW2019}. The large automorphism group suggest that they might have large determining and distinguishing numbers;  the large vertex stabilizers suggest the opposite.

This paper is organized as follows. In Section~\ref{sec:PXIntro}, we provide a definition of the Praeger-Xu graphs and facts about their automorphism groups. In Section~\ref{sec:PXn1}, we show that most Praeger-Xu graphs are twin-free; for those with twins, we use a quotient graph construction to find the determining and distinguishing number. In Section~\ref{sec:DetNoTwins}, we find the determining number for twin-free Praeger-Xu graphs. As a tool for computing distinguishing number, in Section~\ref{sec:Interchangeable} we characterize pairs of vertices in twin-free Praeger-Xu graphs that are interchangeable via an automorphism. Finally, in Section~\ref{sec:DistNoTwins} we show that all twin-free Praeger-Xu graphs are $2$-distinguishable and compute the  cost of $2$-distinguishing. Our results are summarized in Table~\ref{tab:summary}.

\begin{table}[]
    \centering
    \begin{tabular}{|c|c|c|c|c|} \hline
          Parameter & Value &  Condition(s) \\ \hline
      $\det(\px(n,k))$ & 6 & $(n,k) = (4,1)$ \\ \cline{2-3}
      & $\lceil\frac{n}{k}\rceil$ & $k \neq \frac{n}{2}$ but $(n,k) \neq (4,1)$ \\ \cline{2-3}
      & $\lceil\frac{n}{k}\rceil+1 = 3$ & $k = \frac{n}{2}$\\
      \hline
      $\dist(\px(n,k))$ & 5 & $(n,k) = (4,1)$ \\ \cline{2-3}
      & $3$ & $n \neq 4$, $k=1$\\ \cline{2-3}
      & $2$ & $k \ge 2$\\
      \hline
      $\rho(\px(n,k))$ & $5$ & $(n,k) = (4,2)$  \\ \cline{2-3}
       $(k \ge 2)$& $\lceil\frac{n}{k}\rceil$ & $5 \le n < 2k$, or\\ 
       & &  $2k<n \text{ and }  n \notin \{0 \bmod k, -1 \bmod k\}$ \\\cline{2-3}
        & $\lceil\frac{n}{k}\rceil + 1$ &  
        otherwise\\
      \hline     
    \end{tabular}
    \caption{Summary of Symmetry Parameters ($n \ge 3$)}
    \label{tab:summary}
\end{table}

\section{Praeger-Xu Graphs, $\px(n,k)$}\label{sec:PXIntro}

In  1989, Praeger and Xu \cite{SGOTPV} introduced a family of connected graphs they denoted by $C(m,r,s)$, where $m\ge 2$, $r \ge 3$ and $s \ge 1$, that are vertex-transitive for $r \ge s$ and arc-transitive,  hence symmetric, for $r \ge s+1$. This was part of an investigation into connected, symmetric graphs whose automorphism groups have the property that for any vertex $v$, the subgroup of automorphisms fixing $v$ (the stabilizer of $v$) does not act primitively on the set of neighbors of $v$. 
The Praeger-Xu graphs are those where $p=2$; the notation $PX(n,k)$ denotes $C(2, n, k)$.
There are several ways of describing Praeger-Xu graphs (see \cite{GP1994}); we use what is called the bitstring model. 

\begin{defn}\label{defn:pxnk}
Let $n\geq3$ and $1\leq k<n$. The corresponding 
the Praeger-Xu graph is $\px(n,k)=(V,E)$, where $V$ is the set of all ordered pairs $(i,x)$, where $i\in\Z_n$ and $x= x_0x_1 \cdots x_{k-1} $ is a bitstring of length $k$,
and $\{(i,x),(j,y)\}\in E$ if and only if $j=i+1$ and $x=az_1z_2\cdots z_{k-1}$ and $y=z_1z_2\cdots z_{k-1}b$ for some $z_1,\dots,z_{k-1},a,b\in \Z_2$.
\end{defn}

Throughout this paper, subscripts on bits will be considered elements of $\mathbb Z_k$. 
We say that the bit $x_j$ in $x$ is {\it flipped} if it is switched to $x_j+1$ in $\mathbb Z_2$.

There is a natural partition of $V$ into \emph{fibres}  $\F_i=\{(i,x): x\in \mathbb Z_2^k\}$ for  each $i\in\Z_n$. 
Each fibre is an independent set of $2^k$ vertices; every vertex in $\F_i$ is adjacent to exactly two vertices in each of $\F_{i+1}$ and $\F_{i-1}$, so $\px(n,k)$ is $4$-regular, or tetravalent. Two fibres $\F_i$ and $\F_j$ are {\it antipodal} if and only if $n$ is even and $i-j = \frac{n}{2} \bmod n$.

Two Praeger-Xu graphs are illustrated in Figure~\ref{fig:EZsamples}. Figure~\ref{fig:px32} shows the smallest Praeger-Xu graph having  $k>1$, namely $\px(3,2)$, of order $3 \cdot 2^2 = 12$. Figure~\ref{fig:px205} shows the larger Praeger-Xu graph $\px(20,5)$ of order $20\cdot 2^5 = 640$. In all our diagrams of Praeger-Xu graphs, $\F_0$ is the fibre in the $12$ o'clock position, with remaining fibres labeled consecutively clockwise. The vertices in $\F_0$ on $\px(3,2)$ have been labeled with their bitstring components; the bitstring components of vertices in $\F_1$ and $\F_2$ follow the same pattern. More generally, the bitstring components are the binary representations of the integers $0$ to $2^k$, starting with the innermost vertex. 
Throughout this paper, we will be assuming that $n \ge 3$ and $1 \le k < n$, unless otherwise explicitly indicated. 

\begin{figure}[h]
\begin{subfigure}{0.5\textwidth}
\includegraphics[scale=0.28, center]{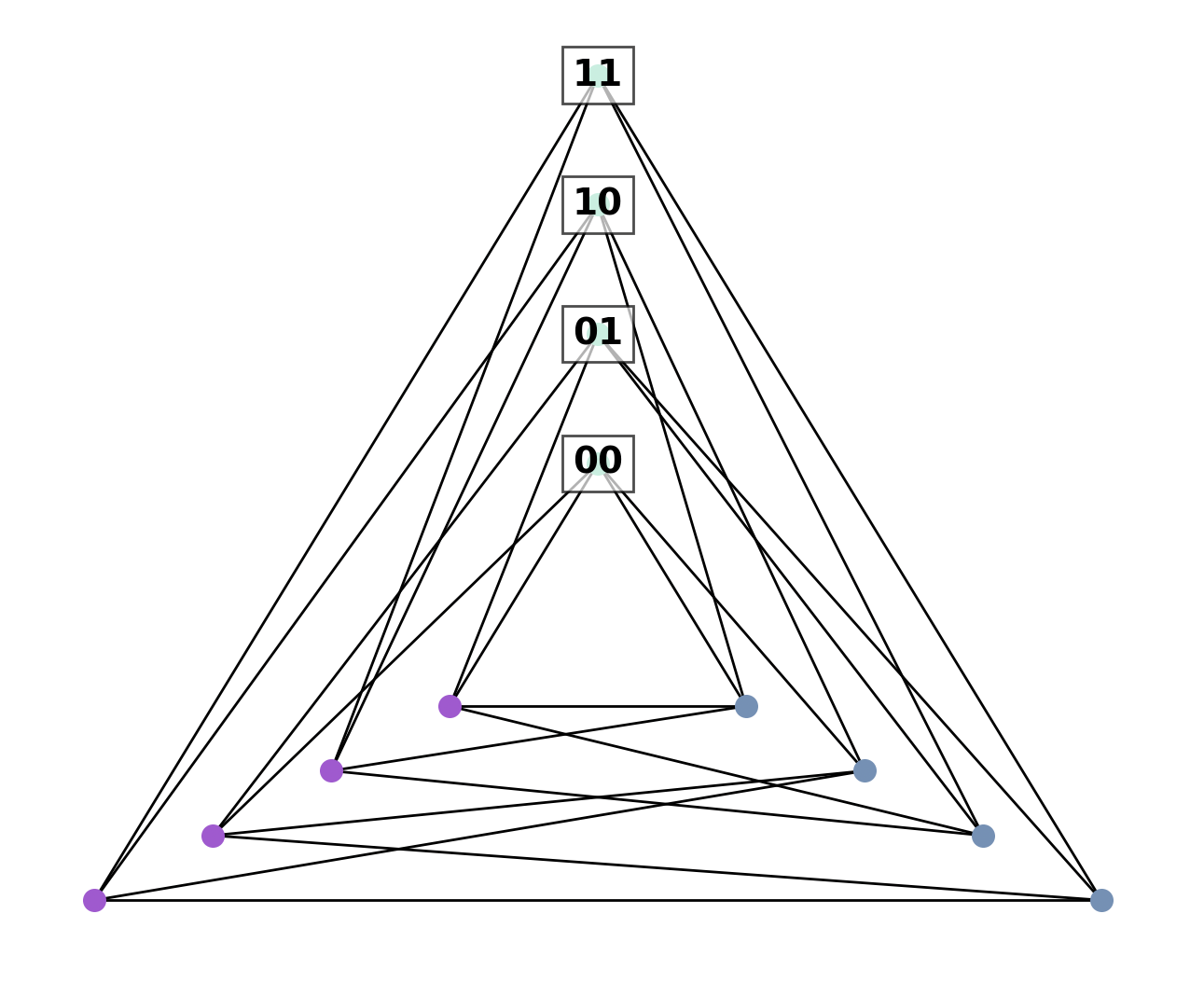}
\caption{$\px(3,2)$.}
\label{fig:px32}
\end{subfigure}
\begin{subfigure}{0.5\textwidth}
\includegraphics[scale=0.16, center]{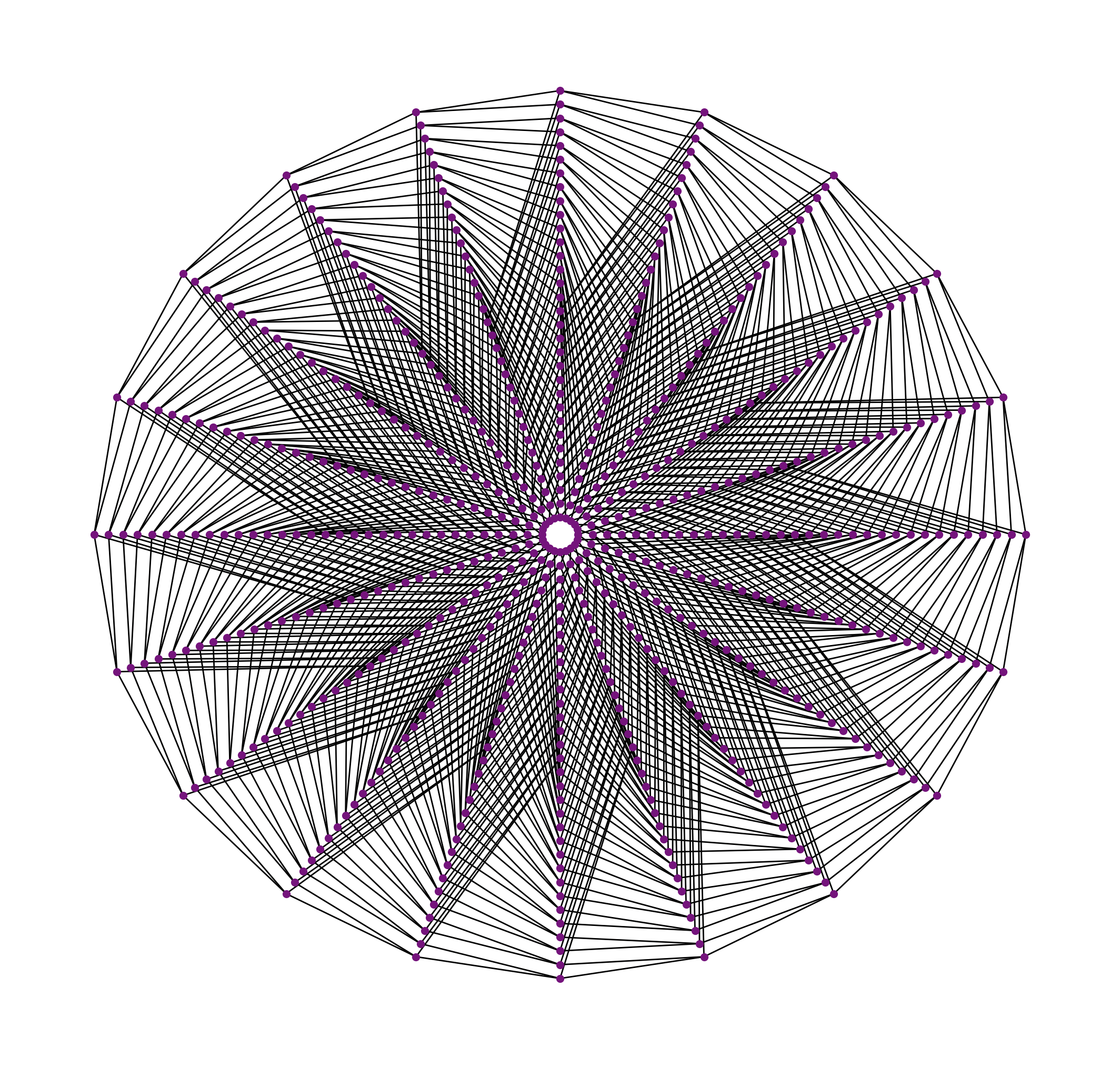}
\caption{$\px(20,5)$.}
\label{fig:px205}
\end{subfigure}
\caption{Two Praeger-Xu graphs.}
\label{fig:EZsamples}
\end{figure}

\subsection{Automorphisms of $\px(n,k)$}

In \cite{SGOTPV}, Praeger and Xu described the automorphism groups of all graphs in the family $C(p, r, s)$. We will adopt the notation used in \cite{OTCOPXG} for automorphisms of the Praeger-Xu graphs, $PX(n,k)$. The automorphism group is generated by three different types of automorphisms. 

 The first is the rotation $\rho$, defined by 
$\rho\cdot(i,x)=(i+1,x).$
Composing $\rho$ with itself $s$ times 
corresponds to a rotation by $s$ fibres: $\rho^s\cdot(i,x)=(i+s,x)$. If $s$ is a multiple of $n$, then the resulting map is the identity, and so we can interpret $s$ as an element of $\mathbb Z_n$.

The second automorphism is the reflection defined by $\mu\cdot(i,x) = (-i, x^-),$ where $x^- = (x_0x_1 \cdots x_{k-1})^- = x_{k-1} \cdots x_1x_0$.  
It is easily verified that $\mu^2 = $ id and $\mu \rho \mu = \rho^{-1}$, so the subgroup $\langle \rho, \mu\rangle$ of $\aut(\px(n,k))$ is the dihedral group $D_n$.
 
 Following \cite{OTCOPXG}, 
for each $s\in\Z_n$ we let $\mu_s = \rho^{s+1-k} \mu \in \langle \rho, \mu\rangle$, so that
$\mu_s\cdot(i,x)=(s+1-k-i,x^-).$
With this notation, $\mu=\mu_{k-1}$; in particular, note that  $\rho^0=\text{id}$ but $\mu_0\neq\text{id}$. We collect some elementary facts about the reflections $\mu_s$ in the following lemma.

 \begin{lem}\label{lem:ijReflec}
Let $s,i,j\in\Z_n$. 
\begin{enumerate}[(1)]
\item The reflection $\mu_s$ interchanges fibres  $\F_i$ and $\F_{s+1-k-i}$; equivalently, fibres $\F_i$ and $\F_j$ are interchanged by $\mu_{i+j+k-1}$.
\item If $n$ is odd, then each $\mu_s$ preserves exactly one fibre. If $n$ is even and $s= k \bmod 2$, then $\mu_s$ does not preserves any fibre, and if $s \neq k \bmod 2$, then $\mu_s$ preserves exactly two antipodal fibres. 
\end{enumerate}
\end{lem}

 \begin{proof}
The proof is straightforward and left to the reader.
 \end{proof}

 The third type of automorphism is, for each $s\in \Z_n$, defined by 
$$\tau_s\cdot(i,x)=\begin{cases} 
(i,x^{s-i}), \quad &
\text{if } i\in \{s, s-1, s-2, \dots, s-k+1\},\\ (i,x), & \text{otherwise,}\end{cases}
$$
where $x^j$ denotes the bitstring $x$ with bit $x_j$ flipped. Thus $\tau_s$ flips bit $x_{s-i}$ of the bitstring component of every vertex in $\F_i$ if $i\leq s\leq i+k-1$, and acts trivially on $\F_i$ otherwise.
Equivalently, vertices in $\F_i$ have their bitstring components altered only by $\tau_i, \tau_{i+1}, \dots, \tau_{i+k-1}$. 
Clearly each $\tau_s$ has order $2$ and  $\tau_s,\tau_t$ commute for all $s, t \in \mathbb Z_n$. Hence the subgroup of $\aut(\px(n,k))$ generated by these automorphisms satisfies  
$K=\langle \tau_0,\tau_1,\tau_2,\dots,\tau_{n-1}\rangle\simeq \Z_2^n$. Each $\tau\in K$ can be represented by 
$$\tau=\tau_0^{u_0}\tau_1^{u_1}\tau_2^{u_2}\cdots\tau_{n-1}^{u_{n-1}},$$ 
where $u_m\in\{0,1\}$ for each $m\in\Z_n$. It is easy to verify that $\rho^{-1}\tau_s \rho = \tau_{s+1}$ and $\mu \tau_s \mu = \tau_{k-1-s}$, so $K$ is a normal subgroup of the group generated by $\rho, \mu$ and $\tau_0, \dots \tau_{n-1}$.

Let $\mathcal{A}= K \rtimes \langle\rho,\mu\rangle = K \rtimes D_n$. Then if $\alpha\in\mathcal{A}$, $\alpha=\tau\delta$ for some $\tau\in K$ and $\delta\in\langle\rho,\mu\rangle = D_n$.
Praeger and Xu showed in  \cite{SGOTPV} that for all $n\neq 4$, $\mathcal{A}=\aut(\px(n,k))$, while for $n=4$, $\mathcal{A}$ is a proper subgroup of $\aut(\px(4,k))$. 

Note that under any $\alpha \in \mathcal A$, vertices in the same fibre  will  be mapped to vertices in the same fibre. In other words, the fibres form a block system for the action of $\mathcal A$ on $\px(n,k)$. 
From \cite{OTCOPXG}, the induced action of $\alpha = \tau\delta\in \mathcal A$ on the fibres of $\px(n,k)$ is $\alpha(\F_i) = \F_{\delta(i)}$, where $\delta(i)$ is simply the action of the dihedral group element  $\delta \in  D_n$ on $i \in V(C_n) = \mathbb Z_n$.
Since any $\tau=\tau_0^{u_0}\tau_1^{u_1}\tau_2^{u_2}\cdots\tau_{n-1}^{u_{n-1}} \in K$ preserves fibres, for any 
$\alpha=\tau\delta\in\mathcal{A}$ and $(i,x)\in V$, we have
$$\alpha\cdot (i,x)=\tau\cdot(\delta\cdot (i,x))=\tau\cdot(\delta(i),y)=(\delta(i),z),$$ 
where $y = x$ if $\delta$ is a rotation $\rho^s$, $y = x^-$ if $\delta$ is a reflection $u_s$, and for all $j \in \mathbb Z_k$,
$z_j = y_j + 1$ if $ u_{\delta(i) -j}=1$ and $z_j =
y_j$ otherwise.

\section{Determining and Distinguishing $\px(n,1)$}\label{sec:PXn1}

For any vertex $v$ in a graph $G=(V,E)$, the (open) {\it  neighborhood} of $v$ is $N(v) = \{u : \{u,v\} \in E\}$. Distinct vertices $x$ and $y$ are (nonadjacent) {\it twins} if and only if  $N(x)=N(y)$. Twins are relevant to notions of graph symmetry because if $x$ and $y$ are twins, then the map that interchanges $x$ and $y$ and fixes all other vertices is a graph automorphism. 

\begin{thm}\label{thm:NoTwin}
For $k=1$, two distinct vertices in $\px(n,1)$ are twins if and only if either they are in the same fibre, or $n=4$ and they are in antipodal fibres.
For $k\geq2$, $\text{PX}(n,k)$ is twin-free.
\end{thm}

\begin{proof}
First assume $k=1$. For any $i\in\Z_n$, 
$\F_i=\{(i,0),(i,1)\}$. By definition, 
\[N((i,0))=\{(i+1,0),(i+1,1),(i-1,0),(i-1,1)\} = \F_{i+1} \cup \F_{i-1} = N((i,1)).\]
Hence $(i,0)$ and $(i,1)$ are twins. More generally,  $(i,x)$ and $(j,y)$ are twins if and only if $\F_{i+1} \cup \F_{i-1} = \F_{j+1} \cup \F_{j-1}$, or equivalently, as subsets of $\mathbb Z_n$, $\{i+1, i-1\} = \{j+1, j-1\}$. If $i \neq j$, then  $i+1 = j-1$ and $i-1 = j+1$. This implies $i-j = 2 = -2$in $\mathbb Z_n$. Since $n \ge 3$, we can conclude $n=4$ and $\F_i$ and $\F_j$ are antipodal fibres.

Next assume $k\ge 2$.
Let $u$ and $v$ be distinct vertices in $\px(n,k)$ such that $N(u)=N(v)$. 
Let $u=(i,axb)$ and $v=(j,cyd)$ for some  $i, j \in \mathbb Z_n$, $a, b, c, d \in \{0, 1\}$, and $y, x \in \mathbb Z_2^{k-2}$ (where $y$ and $x$ are empty strings if $k=2$).
By definition, $N(u)=\{(i+1,xb0),(i+1,xb1),(i-1,0ax),(i-1,1ax)\}$
and 
$N(v)=\{(j+1,yd0),(j+1,yd1),(j-1,0cy),(j-1,1cy)\}.$
Since $N(u)$ consists of two vertices in each of $\F_{i+1}$ and $\F_{i-1}$, and $N(v)$ consists of two vertices in each of $\F_{j+1}$ and $\F_{j-1}$,  $\{i+1, i-1\} = \{j+1, j-1\}$. 

Suppose  $i= j \bmod n$. 
Comparing neighbors in $\F_{i+1} = \F_{j+1}$ with the same final bit gives $xb0=yd0$ and $xb1=yd1$. Hence $xb=yd$ in $\mathbb Z_2^{k-1}$.
An analogous argument can be used in $\F_{i-1}$ to show that $ax=cy$ in $\mathbb Z_2^{k-1}$. Thus $axb=cyd$ in $\mathbb Z_2^k$. Since $i= j$ in $\mathbb Z_n$, $(i,axb)=(j,cyd)$ and so $u=v$, contradicting the assumption that $u$ and $v$ are distinct.

Alternatively, if $i\neq j \bmod n$, then as argued earlier in this proof, $n=4$ and $i-1 =j+1 \bmod n$. 
Hence $N(u)\cap\F_{i-1}= N(v) \cap \F_{j+1}$, so
\[
\{(i-1,0ax),(i-1,1ax)\} = \{(j+1, yd0), (j+1, yd1)\}.
\]
Since $x$ and $y$ cannot be simultaneously both $0$ and $1$, this is impossible.
\end{proof}

Figure~\ref{fig:EZsamples} depicts two Praeger-Xu graph with twins.  Note that every vertex of $\px(4,1)$ is in a set of $t=4$ mutual twins, while for any $n\ge 3, n\neq 4$, every vertex of $\px(n,1)$ is in a set of $t=2$ mutual twins.

\begin{figure}[h]
\begin{subfigure}{0.5\textwidth}
\includegraphics[scale=0.25, center]{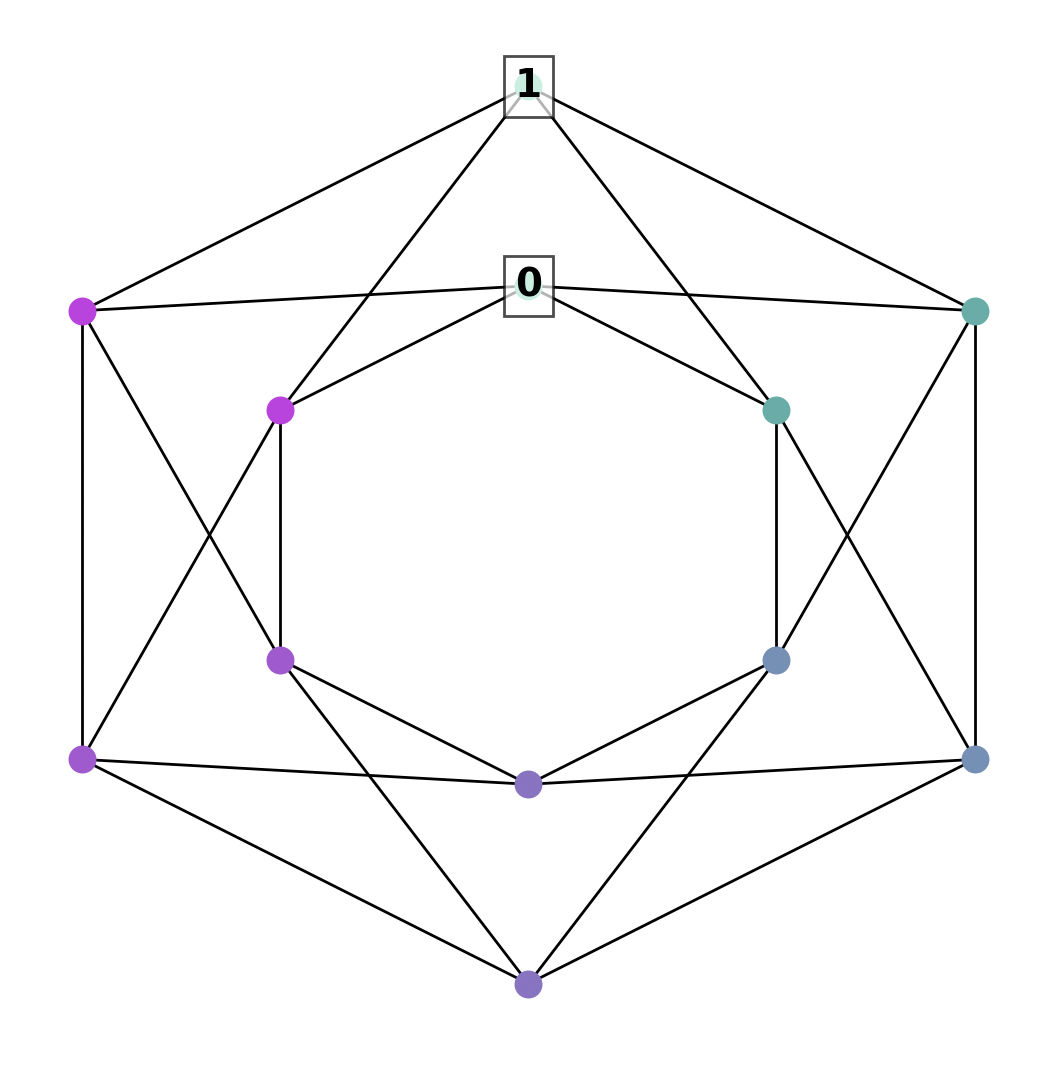}
\caption{$\px(6,1)$ }
\label{fig:px61}
\end{subfigure}
\begin{subfigure}{0.5\textwidth}
\includegraphics[scale=0.2, center]{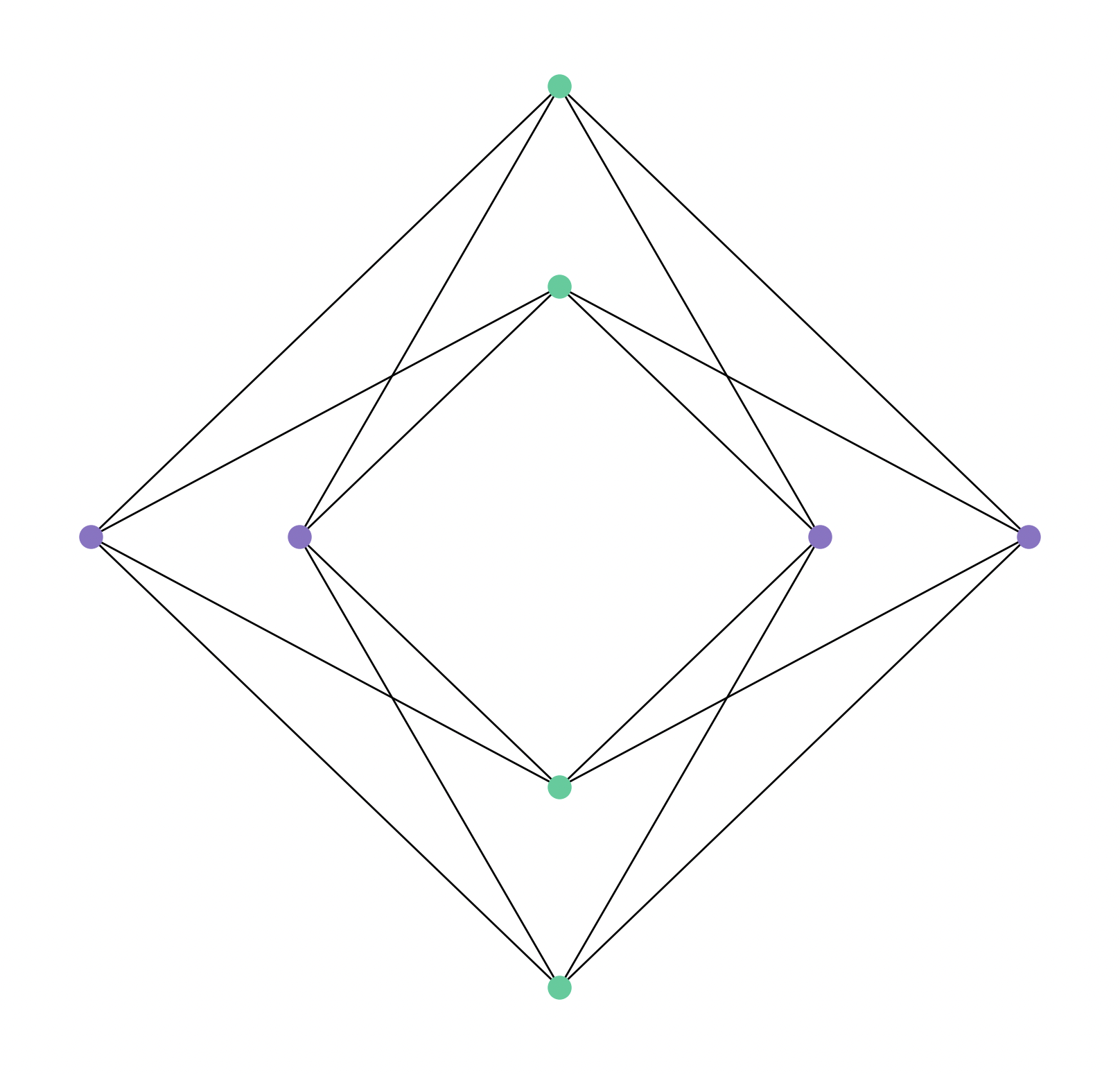}
\caption{$\px(4,1)$.}
\label{fig:px41}
\end{subfigure}
\caption{Two Praeger-Xu graphs with $k=1$.}
\label{fig:EZsamples}
\end{figure}

For any graph $G= (V,E)$, can define an equivalence  relation on $V$ by $x\sim y$ if and only if $x$ and $y$ are twins. 
The corresponding \emph{twin quotient graph} $\widetilde{G}$ has as its vertex set the set of equivalence classes $[x]=\{y\in V(G):x\sim y\}$, with $\{[x],[z]\}\in E(\widetilde{G})$ if and only if there exist $p\in[x]$ and $q\in[z]$ such that $\{p,q\}\in E(G)$. (Note that by definition of the twin relation, $\{[x],[z]\}\in E(\widetilde{G})$ if and only if for all $p\in[x]$ and $q\in[z]$,  $\{p,q\}\in E(G)$.) The symmetry parameters of the twin quotient graph $\widetilde G$ can be used to give information about the symmetry parameters of $G$.

A subset $T\subseteq V(G)$ that contains all but one of the vertices from each equivalence class of twin vertices is called a \emph{minimum twin cover} of $G$.  In 2020, Boutin {\it et al.} proved the following result. 

\begin{thm}\label{thm:DettG} \cite[Theorem 19] {BCKLPR2020b}
 Let $G$ be a graph  with twin quotient graph $\widetilde G$. Let $T$ be a minimum twin cover of $G$ and let 
 $\widetilde T = \{[u] : u \in T\} \subseteq V(\widetilde G)$. Let $\widetilde S$ be a determining set for $\widetilde G$ containing $\widetilde T$. Let $R=\{x\in V(G) : [x]\in \widetilde S\setminus \widetilde T\}$. Then $S=T \cup R$ is a determining set for $G$. Furthermore, if $\widetilde S$ is of minimum size among determining sets for $\widetilde G$ that contain $\widetilde T$, then $S$ is a minimum determining set for $G$.
\end{thm}

For distinguishing  number, we have the following result from Cockburn and Loeb. 

\begin{thm}\label{thm:DistTwins}\cite[Theorem 2]{CL2023}
Let $G$ be a graph in which every vertex is in a set of $t$ mutual twins.
If $\dist(\widetilde G) = \widetilde d$, then $\dist(G) = d$, where $d$ is the smallest positive integer such that $\binom{d}{t} \ge \widetilde d$.
\end{thm}

These theorems can be used to find the symmetry parameters of all Praeger-Xu graphs with twins, $\px(n,1)$.

\begin{thm}\label{thm:DetkIsOne}
For $n=4$, $\det(\px(4,1)) = 6$ and $\dist(\px(4,1)) = 5$.
For  all $n \neq 4$,
$\det(\px(n,1)) = n$ and $\dist(\px(n,1)) = 3$.
\end{thm}

\begin{proof}
For $n=4$, there are only two equivalence class under the twin relation, namely  $[(0,0)]=\F_0 \cup \F_2$ and $[(1,0)] = \F_1 \cup \F_3$, so the twin quotient graph is $\widetilde{\px}(4,1) = K_2$. In this case, $T = \{(0,0), (0,1), (2,0), (1,0), (1,1), (3,0)\}$ is a minimum twin cover with $\widetilde T = V(\widetilde{\px}(4,1))$. 
The only determining set $\widetilde S$ of the twin quotient graph containing $\widetilde T$ is the entire vertex set. 
Hence by Theorem~\ref{thm:DettG}, $T$ is a minimum determining set and so $\det(\px(4,1)) = 6$. 
Since every vertex of $\px(4,1)$ is in a set of $t=4$ mutual twins and $\dist(K_2)= 2$, by Theorem~\ref{thm:DistTwins}, $\dist(\px(4,1)) = 5$.

Next assume $n \neq 4$. In this case, the equivalence classes under the twin relation are the fibres of $\px(n,1)$ 
 and so $T=\{(i,0):i\in\Z_n\}$ is a minimum twin cover of $\px(n,1)$.
Additionally, every vertex in $V(\px(n,1))$ belongs to one of these fibres, so $V(\widetilde{\px}(n,1))=\widetilde{T}$. Again the only determining set $\widetilde S$ of the twin quotient graph containing $\widetilde T$ is the entire vertex set and so by Theorem~\ref{thm:DettG}, $T$ is a minimum determining set. Thus  $\det(\px(n,1)) = |T| = n$.

Since every fibre in $\px(n,1)$ is a vertex in $\widetilde{\px}(n,i)$, and vertices in $\mathcal{F}_i$ are adjacent only to vertices in $\mathcal{F}_{i+1}$ and $\mathcal{F}_{i-1}$, $\widetilde{\px}(n,1)= C_n$. 
Hence, $\widetilde{d}=\text{Dist}(\widetilde{\px}(n,1))=\text{Dist}(C_n)=3$ if $n \in \{3, 5\}$, and $\widetilde{d}=2$ if  $n\geq 6$. 
Since each vertex is in a set of $t=2$ twins, by Theorem~\ref{thm:DistTwins},
$d$ is the smallest integer such that ${d\choose2} \geq3$ if $n \in \{3, 5\}$, and  
$d$ is the smallest integer such that ${d\choose2}\geq2$ if $n \ge 6$. In both cases, $d=3$.
\end{proof}

\section{Determining $\px(n,k)$, $k \ge 2$}\label{sec:DetNoTwins}

In this section, we find the determining number for twin-free Praeger-Xu graphs. Recall from  Section~\ref{sec:PXIntro} that for $n \neq 4$, $\mathcal A = K \rtimes \langle \rho, \mu \rangle=\aut(\px(n,k))$, whereas for $n=4$, $\mathcal A$ is a proper subgroup of $\aut(\px(4,k))$. We begin with a lemma that applies to all Praeger-Xu graphs and apply it to the general case $n \neq 4$. We then consider the exceptional cases $\px(4,2)$ and $\px(4,3)$.

\begin{lem}\label{lem:tauPreservesEven} 
Let $i\in\Z_n$, $S_i\subseteq \F_i$
and $\tau = \tau_0^{u_0} \cdots \tau_{n-1}^{u_{n-1}}\in K$. If $\tau(S_i)=S_i$ and $|S_i|$ is odd, then $\tau$ acts trivially on $\F_i$; equivalently,
\[u_i = u_{i+1} = 
\dots = u_{i+k-1} = 0.\]   
\end{lem}

\begin{proof} 
Assume $\tau(S_i)=S_i$ and that $\tau$ acts nontrivially on $\F_i$. Let $s\in S \subseteq \F_i$, so by assumption, $\tau\cdot s\in S_i$.
Since every element in $K\simeq \Z_2^n$ has order $2$, 
$\tau\cdot(\tau\cdot s)=\tau\inv\cdot(\tau\cdot s) = s.$ 
Additionally, since $\tau$ acts nontrivially on every vertex of  $\F_i$, $s\neq \tau\cdot s$. Thus, $S$ can be partitioned into pairs of the form $\{s, \tau \cdot s\}$, implying that $S_i$ has an even number of vertices in total.
\end{proof}

\begin{thm}\label{thm:DetnNot4} 
For $n\neq4$,
$$\det(\px(n,k))=\begin{cases}\lceil \frac{n}{k}\rceil, \quad & \text{ if } k\neq\frac{n}{2},\\ \lceil \frac{n}{k}\rceil + 1 = 3, & \text{ if } k=\frac{n}{2}.\end{cases}$$
\end{thm}
\begin{proof}
First suppose $k\neq\frac{n}{2}$. 
Let $S\subset V(\px(n,k))$ such that $|S|=\lceil\frac{n}{k}\rceil-1$, and assume $S$ is a determining set for $\px(n,k)$.
The set of indices of the fibres containing elements of $S$ is $$I_S=\{i\in\Z_n\mid S\cap\F_i\neq\emptyset\}=\{i_1,i_2,\dots,i_s\},$$ where  $0\le i_1 < i_2 \dots < i_s \le n-1$. Then the number of fibres in the gaps between these fibres are $i_2-i_1-1,i_3-i_2-1,\dots,n+i_1-i_s-1$. 
If  $i_{p+1}-i_p-1\geq k$ for some $i_p, i_{p+1} \in S$ , then
$\tau_{i_{p+1}-1}$ is a nontrivial automorphism that fixes $S$, contradicting the assumption that $S$ is a determining set.  Thus each gap contains at most $k-1$ fibres.
Since every fibre either contains a vertex in $S$ or is in a gap between two such fibres, the total number of fibres satisfies
$$|I_S|+|I_S|(k-1)=|I_S|k\ \le |S| k =(\lceil\tfrac{n}{k}\rceil-1)k<n,$$
a contradiction.
Thus, $\det(\px(n,k))>\lceil\frac{n}{k}\rceil-1$.

We claim 
$
S=\big\{v_{ik}\in\F_{ik}:i\in\{0,1,\dots,\lceil\tfrac{n}{k}\rceil-1\}\big\},
$
where $v_{ik}$ is any vertex in $\F_{ik}$,
is a determining set for $\text{PX}(n,k)$. 
Let $\alpha=\tau\delta\in\mathcal A=\aut(\px(n,k))$ such that $\alpha$ fixes every vertex in $S$.
Since $k\neq \tfrac{n}{2}$, $\F_0$ and $\F_k$ are non-antipodal fibres. Since the induced action of $\alpha$ on the fibres is an element of $D_n$ that fixes non-antipodal vertices  $0$ and $k$ in $C_n$, $\delta=\text{id}$.

Next we show that $\tau=\tau_0^{u_0}\tau_1^{u_1}\cdots\tau_{n-1}^{u_{n-1}}=\text{id}$.
Let $S_0 = S\cap \F_0$; note that $|S_0|=1$ is odd. Since $\tau$ fixes every vertex in $S$, $\tau(S_0) = S_0,$
and so by Lemma~\ref{lem:tauPreservesEven}, $u_0 = u_1 = \dots = u_{k+1} = 0$. Applying the same logic to $S_k = S\cap \F_k$, we get $u_k = u_{k+1} = \dots = u_{2k-1} = 0$. Iterating this argument for $S_{ik}$ for all $i \in \{0, 1, \dots \lceil\frac{n}{k}\rceil - 1\}$, we conclude $u_0 =\dots= u_{n-1}=0$. Thus $\tau = \text{id}$. By definition, $S$ is a determining set and so $\det(\px(n,k))\leq |S|=\lceil\frac{n}{k}\rceil$.

Now suppose $k=\frac{n}{2}$. 
Assume $S$ is a determining set of cardinality $2$. Since $\px(n,k)$ is vertex-transitive,  
we can assume without loss of generality that $S=\{z=(0,00\cdots 0),v= (i,x)\}$. There are three cases: $i=0$, $i=k=\frac{n}{2}$, or $i\neq 0$ and $i\neq k=\frac{n}{2}$.

For the first case, assume $i=0$, so both $z,v\in\F_0$.
Then since $\tau_{n-1}$ affects $\F_i$ if and only if $i\in\{k=\frac{n}{2},\frac{n}{2}+1,\dots,n-1\}$, and $0$ is not in that set, $\tau_{n-1}$ is a nontrivial automorphism that fixes $S$, a contradiction. 

For the second case, assume $i=k=\frac{n}{2}$. Then $z$ and $v$ are in antipodal fibres. 
If we we apply the reflection $\mu$ to $S$, we get 
$$
\mu\big(S\big)=\big\{(0,(00\cdots0)^-),(-k,x^-)\big\}=\big\{(0,00\cdots0),(k,x^-)\big\}.
$$
For each $m$ such that $x_m\neq (x^-)_m$, 
the automorphism $\tau_{k+m}\in K$  flips this bit in the bitstring component of every vertex in $\F_k$, but has no effect on the vertices in $\F_0$. 
Let $\tau=\tau_0^{u_0}\tau_1^{u_1}\cdots\tau_{n-1}^{u_{n-1}}$ where $u_{k+m} = 1$ if $x_m \neq (x^-)_m$ and $0$ otherwise. Then $\zeta = \tau \mu$ fixes both $z$ and $v$, contradicting our assumption that $S$ is a determining set.

In the third case, $i\neq0$ and $i\neq k =\frac{n}{2}$, so we can assume that in $\mathbb Z$, either $0<i<k$ or $k < i < n$.
In the first case,  $\tau_{n-1}$ fixes both $z$ and $v$, and in the second,  $\tau_{i-1}$ fixes both $z$ and $v$. Hence, $S$ is not a determining set for $\px(n,k)$.
As we have covered all possible cases, we conclude $\det(\px(n,k))>2$.

Finally, let $S = \{v_0, v_1, v_k\}$ where $v_0\in \F_0$, $v_1 \in \F_1$ and $v_k \in \F_k$ and assume $\alpha = \tau \delta$ fixes $S$. The induced action of $\alpha$ on the fibres corresponds to an element of the dihedral group that fixes non-antipodal vertices $0$ and $1$ in $C_n$, so $\delta = \text{id}$. Next,  $\tau = \tau_0^{u_0} \cdots \tau_{n-1}^{u_{n-1}}$ fixes one vertex in each of $\F_0$ and $\F_k$, so by Lemma~\ref{lem:tauPreservesEven}, $u_0 = \dots = u_{k-1} = u_k = \dots = u_{2k-1}= u_{n-1} = 0$. Thus $\tau = \text{id}$. By definition, $S$ is a determining set for $\px(n,k)$ so $\det(\px(n,k))\leq|S|=3$. Thus $\det(\px(n,k))=3$.
\end{proof}

We now turn our attention to the exceptional cases $\px(4,2)$ and $\px(4,3)$. It is stated without proof in \cite{SGOTPV} that $Q_4\cong\px(4,2)$; 
we provide an explicit isomorphism.
A bitstring $x$ is \emph{even} if $x$ has an even number of ones, and  $x$ is \emph{odd} otherwise.
Define $\varphi:V(Q_4)\rightarrow V(\px(4,2))$ by $$\varphi(x_0x_1x_2x_3)=\begin{cases}(j,x_1x_3), \quad &\text{ if } j\text{ is odd},\\ (j,x_3x_1), &\text{ if } j\text{ is even},\end{cases}$$
where $$j=\begin{cases}
0, \quad & \text{ if } x_0x_1\text{ and } x_2x_3 \text{ are both odd},\\
1, & \text{ if } x_0x_1\text{ is odd and } x_2x_3\text{ is even},\\
2, & \text{ if } x_0x_1\text{ and } x_2x_3\text{ are both even},\\
3, & \text{ if } x_0x_1\text{ is even and } x_2x_3\text{ is odd}.
\end{cases}$$ 

Figure~\ref{fig:Q4PX42} is a drawing of $Q_4$, with vertices positioned as they would be in a canonical drawing of $\px(4,2)$, as explained at the beginning of Section~\ref{sec:PXIntro}.

\begin{figure}[h]
\includegraphics[scale=0.35, center]{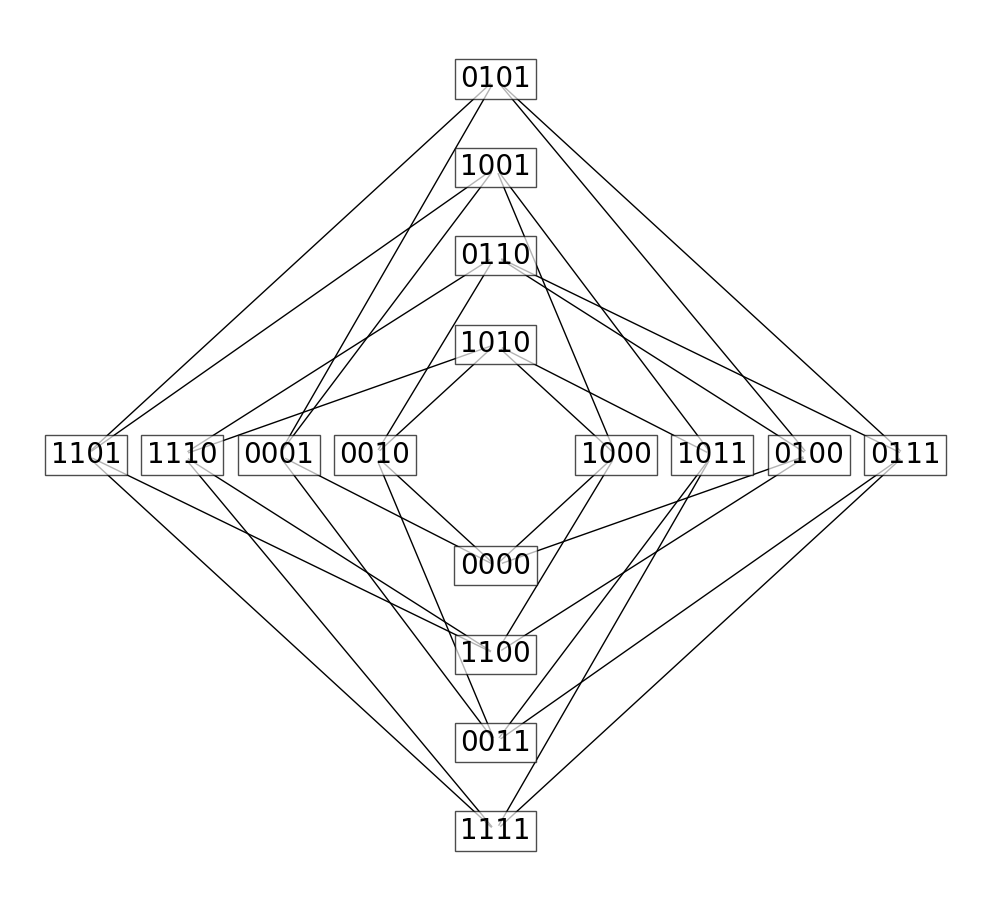}
\caption{$Q_4$ is isomorphic to  $\px(4,2)$.}
\label{fig:Q4PX42}
\end{figure}

\begin{prop}\label{prop:PX42}
$ \det(\px(4,2))=3, \dist(\px(4,2))=2 \text{ and }\rho(\px(4,2))=5.$
\end{prop}
\begin{proof}
This follows immediately from previous work on the symmetry parameters of $Q_4$.
By Theorem 3 from \cite{Bo2009a}, $$\det(\px(4,2))=\det(Q_4)=\lceil\log_24\rceil+1=2+1=3.$$ Notably, this expression agrees with the formula given in  Theorem~\ref{thm:DetnNot4} because $\lceil\frac{4}{2}\rceil+1=2+1=3$. By Theorem 5 from \cite{BoCo2004}, $\dist(\px(4,2))=\dist(Q_4)=2$; by Theorem 11 from \cite{Bo2021}, $\rho(\px(4,2))=\rho(Q_4)=5$.
\end{proof}

\begin{prop}\label{prop:detFourThree}
$\det(\px(4,3)) =2 = \lceil \frac{4}{3} \rceil.$   
\end{prop}

\begin{proof}
 From \cite{SGOTPV}, $\mathcal A =K\rtimes \langle\rho,\mu\rangle$  is a proper subgroup of $\aut(\px(4,3))$ of index $2$. The proof of Theorem~\ref{thm:DetnNot4} can be used when $(n,k) = (4,3)$ to show that the only $\alpha \in \mathcal A$ that fixes 
 two vertices from non-antipodal fibres is the identity. However, more care must be taken when choosing the two vertices. For example, there is a nontrivial automorphism $\xi \in \aut(\px(4,3))$ that fixes both elements of $S = \{(0,000), (3,000)\}$. 
 More precisely, as a permutation $\xi$ is the product of the disjoint $2$-cycles in Table~\ref{tab:2CyclesInXi}. 
 Unlike for any $\alpha \in \mathcal A$, the fibres do not constitute a block system for $\xi$. 
 However, if we partition each fibre into vertices whose bitstrings are palindromic ($x = x^-$) and vertices whose bitstrings are nonpalindromic ($x \neq x^-$), these half-fibres constitute a block system for  $\xi$.

 \renewcommand{\arraystretch}{1.2}
 \begin{table}[h]
     \centering
     \begin{tabular}{|c|c c|}
     \hline
        within $\mathcal F_0$, $x = x^-$  &  ((0,010), (0,101)) &\\ \hline
        within $\mathcal F_2$, $ x \neq x^-$
          & ((2,001), (2,110)) &\\ \hline
        between $\mathcal F_0$, $x \neq x^-$ and $\mathcal F_2$, $ x = x^-$ & ((0,001), (2,000)), \,  & ((0,100), (2,010)) \\ & ((0,011), (2,101)),  &  ((0,110), (2,111)) \\ \hline
        within $\mathcal F_1$, $x \neq x^-$  &  ((1,100), (1,011)) &\\ \hline 
        within $\mathcal F_3$, $ x = x^-$
          & ((3,010), (3,101)) &\\ \hline
        between $\mathcal F_1$, $x = x^-$ and $\mathcal F_3$, $ x \neq x^-$ & ((1,000), (3,100)), & ((1,010), (3,001))\\ & ((1,101), (3,110)), & ((1,111), (3,011)) \\ \hline
     \end{tabular}
     \caption{Disjoint $2$-cycles of $\xi \in \aut(\px(4,3))$.}
     \label{tab:2CyclesInXi}
 \end{table}
 
Implicitly $\xi$ has eight fixed points, with two in each fibre, namely
$(0,000)$, $(0,111)$, $(1,001)$, $(1,110)$, $(2,100)$, $(2,001)$,  $(3,000)$ and $(3,111)$.
Since $\mathcal A$ has index $2$ in $\aut(\px(4,3))$, every automorphism of $\px(4,3)$ is in one of the two cosets, $\mathcal A$ and $\mathcal A \xi$.

 Next we show that $S' = \{(0,000), (3,001)\}$ is a determining set. The two vertices in $S'$ are from non-antipodal fibres, so no nontrivial automorphism in $\mathcal A$
 fixes both vertices in $S'$. Table~\ref{tab:2CyclesInXi} shows that $\xi$ clearly does not fix $(3,001)$; we must also show that no other automorphism in the coset $\mathcal A \xi$ fixes $S'$.

 Assume there exists $\beta \in \mathcal A$ such that $\beta \circ \xi$ fixes $S'$.
 Then 
 $(0,000) = \beta \circ \xi \cdot (0,000) = \beta\cdot (0,000)$ and 
$(3,001) = \beta \circ \xi \cdot (3,001) = \beta \cdot (1, 010)$.
The induced action of $\beta$ on the fibres fixes $\mathcal F_0$ and takes $\mathcal F_1$ to $\mathcal F_3$, so by Lemma~\ref{lem:ijReflec},
$\beta = \tau \circ \mu = \tau_0^{u_0} \tau_1^{u_1} \tau_2^{u_2} \tau_3^{u_3} \mu$.
Since $\beta$ and $\mu$ both fix $(0,000)$, so must $\tau$ and by Lemma~\ref{lem:tauPreservesEven}, $u_0=u_1=u_2 = 0$. However, because $\tau_3$ can only affect the $0$-th bit of the bitstring component of a vertex in $\F_3$, no value of $u_3$ satisfies  
$(3,001) = \beta \cdot (1,010) = \tau_3^{u_3} \mu \cdot (1,010) = \tau_3^{u_3} \cdot (3,010)$.
\end{proof}

The following theorem summarizes our results on the determining number of twin-free Praeger-Xu graphs.

\begin{thm}\label{thm:DET}
For all $n \ge 3$ and $2 \le k <n$,
$$\det(\px(n,k))=
\begin{cases}
\lceil\frac{n}{k}\rceil, \quad &\text{ if } k\neq \frac{n}{2},\\
\lceil\frac{n}{k}\rceil+1=3, &\text{ if } k=\frac{n}{2}.
\end{cases}$$
\end{thm}

\section{Interchangeable Vertices in $\px(n,k)$}\label{sec:Interchangeable}

As mentioned in Section~\ref{sec:PXn1}, if two vertices in a graph are twins, then the map that interchanges them and leaves all vertices fixed is a graph automorphism. By Theorem~\ref{thm:NoTwin}, if $k \ge 2$, then $\px(n,k)$ is twin-free, but we will find it useful to identify when two vertices can be interchanged by an automorphism, regardless of its action on other vertices.

\begin{defn}
Distinct vertices $u, v$ in a graph $G$ are interchangeable if and only if there exists $\alpha \in \aut(G)$ such that $\alpha\cdot u = v$ and $\alpha \cdot v = u$. \end{defn}

There are some situations where it is easy to find  an automorphism interchanging vertices $u=(i,x) $ and $v=(j,y)$ in $\px(n,k)$.
If $i=j$, there exists $\tau=\tau_0^{u_0} \cdots \tau_{n-1}^{u_{n-1}}\in K$ that flips exactly the right bits in bitstring components of vertices in $\F_i$. More precisely, for each $t \in \{0, 1, \dots, k-1\}$, if $x_t \neq y_t$, set $u_{i+t}=1$, and otherwise set $u_{i+t}=0$. The values of $u_m$ for any $m\in \mathbb Z_n$ not of the form $i+t$ do not affect the action of $\tau$ on $u$ and $v$.
If $i\neq j$, then we can find $\delta \in \langle \rho, \mu\rangle$ such that the induced action of $\delta$ on the fibres interchanges $\F_i$ and $\F_j$; we can then look for  $\tau=\tau_0^{u_0} \cdots \tau_{n-1}^{u_{n-1}}\in K$ so that $\tau$ flips exactly the right bits in both $\F_i$ and $\F_j$ to ensure that $\tau\delta$ interchanges $u$ and $v$. If $\F_i$ and $\F_j$ are far enough apart, then we can set the values $u_i, u_{i+1}, \dots u_{i+k-1}$ and $u_j, u_{j+1}, \dots u_{j+k-1}$ independently. However, if $M =\big\{i,i+1,i+2,\dots,i+k-1\big\}\cap\big\{j,j+1,j+2,\dots,j+k-1\big\} \neq \emptyset$, then for any $m \in M$, $\tau_m$ affects both vertices in $\F_i$ and $\F_j$ and there is potential for conflict. 

\begin{lem}\label{lem:interchangeable}
Let  $u=(i,x),v=(j,y)\in V(\px(n,k))$ and let 
\[
M =\big\{i,i+1,i+2,\dots,i+k-1\big\}\cap\big\{j,j+1,j+2,\dots,j+k-1\big\} \subseteq \mathbb Z_n.
\]
Then $u$ and $v$ are interchangeable by some $\alpha \in \mathcal A$
if and only if one of the following holds: 
\begin{enumerate}[(1)]
    \item $j=i$,
    \item $j \neq i$ and for all $m\in M$, $(x^-)_{m-j}=y_{m-j}$ if and only if $(y^-)_{m-i}=x_{m-i}$, 
    \item $j=i+\frac{n}{2}$ and for all $m\in M$, $x_{m-j}=y_{m-j}$ if and only if $y_{m-i}=x_{m-i}$.
\end{enumerate}
\end{lem}

\begin{proof}
Assume $u$ and $v$ are interchangeable by $\alpha = \tau\delta \in\mathcal{A}$, but neither (1) nor (2) holds. Then
$j\neq i$ and  for some $m\in M$, either $(x^-)_{m-j}=y_{m-j}$ but $(y^-)_{m-i}\neq x_{m-i}$,  or $(x^-)_{m-j}\neq y_{m-j}$ but
$(y^-)_{m-i}= x_{m-i}$.
As usual, either $\delta = \rho^s$ or $\delta=\mu_s$ for some $s\in\Z_n$.

If $\delta=\mu_s$, then $\mu_s\cdot u=(j,x^-)$ and $\mu_s\cdot v=(i,y^-)$, and hence $\tau\cdot(j,x^-)=(j,y)$ and $\tau\cdot(i,y^-)=(i,x)$, where $\tau=\tau_0^{u_0}\tau_1^{u_1}\cdots\tau_{n-1}^{u_{n-1}}$. 
If $(x^-)_{m-j}=y_{m-j}$ and $(y^-)_{m-i} \neq x_{m-i}$, then since $\tau_m$ flips the $(m-j)$-th bit of the bitstrings in $\F_j$, $u_m=0$. However, $\tau_m$ flips the $(m-i)$-th bit of bitstrings in $\F_i$, so $u_m=1$, a contradiction. A completely analogous argument works if $(x^-)_{m-j}\neq y_{m-j}$ and $(y^-)_{m-i}= x_{m-i}$. Thus, $\delta=\rho^s$.

Since $\rho^s\cdot u=(i+s,x)$ and $\rho^s\cdot v=(j+s,y)$, and $\tau$ fixes every fibre, $i+s=j$ and $j+s=i$. 
Since we are assuming $i\neq j$, $s=\frac{n}{2}$, so $j=i+\frac{n}{2}$. 
Let $m\in M$, and assume $x_{m-j}=y_{m-j}$. Since $\tau\cdot(j,x)=(j,y)$, $\tau$ must not flip the $(m-j)$-th bit of the bitstrings of $\F_{j}$, so $u_m=0$. That means that $\tau$ must also not flip the $(m-i)$-th bit of the bitstrings of $\F_i$, so since $\tau\cdot(i,y)=(i,x)$, $y_{m-i}=x_{m-i}$. A completely analogous argument works if we assume $y_{m-i}=x_{m-i}$. Thus condition (3) holds.

Conversely, we will show that if one of (1), (2) or (3) holds,
then $u$ and $v$ are interchangeable by some $\alpha\in\mathcal{A}$.
First, assume (1) holds, so $j=i$. 
As noted in the paragraph before the statement of this lemma, there is some $\tau \in K \subset \mathcal A$ that interchanges $u$ and $v$.

Next, assume (2) holds, so $j\neq i$ and for all $m\in M$, $(x^-)_{m-j}=y_{m-j}$ if and only if $(y^-)_{m-i}=x_{m-i}$. 
Let $\tau=\tau_0^{u_0}\tau_1^{u_1}\cdots\tau_{n-1}^{u_{n-1}}\in K$, where for each $s\in\Z_n$, $u_s=1$ if $(x^-)_{s-j}\neq y_{s-j}$ and $(y^-)_{s-i}\ne x_{s-i}$ and $u_s=0$ otherwise. 
Hence, in $\F_j$, $\tau$ flips the bits in every position that $x^-$ and $y$ differ and no others, and in $\F_i$, $\tau$ flips the bits in every position that $y^-$ and $x$ differ and no others.
By Lemma \ref{lem:ijReflec}, there exists $\mu_s\in\langle\rho,\mu\rangle$ such that $\mu_s\cdot u=(j,x^-)$ and $\mu_s\cdot v=(i,y^-)$.
Let $\alpha=\tau\mu_s\in\mathcal{A}$. Then 
$\alpha\cdot u = \tau\cdot(\mu_s\cdot u)= \tau\cdot(j,x^-)=(j,y),$
and
$\alpha\cdot v = \tau\cdot(\mu_s\cdot v)=\tau\cdot(i,y^-)=(i,x),$
so $u$ and $v$ are interchangeable by $\alpha\in\mathcal{A}$.

Lastly, assume (3) holds, so $j=i+\frac{n}{2}$ and for all $m\in M$, $x_{m-j}=y_{m-j}$ if and only if $y_{m-i}=x_{m-i}$.
Let $\tau=\tau_0^{u_0}\tau_1^{u_1}\cdots\tau_{n-1}^{u_{n-1}}\in K$, where for each $s\in\Z_n$, $u_s=1$ if $x_{s-j}\neq y_{s-j}$ and $u_s=0$ otherwise. By assumption, that also means that $u_s=1$ if $y_{s-i}\neq x_{s-i}$ and $u_s=0$ otherwise. Hence, in both $\F_i$ and $\F_j$, $\tau$ flips the bits in every position that $x$ and $y$ differ, and no others. It is straighforward to verify that $\alpha=\tau\rho^{n/2}\in\mathcal{A}$ interchanges $u$ and $v$.
\end{proof}

For example, let $u=(i,x)=(0,101),v=(j,y)=(1,001)\in V(\px(5,3))$. Then $M=\{0,1,2\}\cap\{1,2,3\}=\{1,2\}$. Since $x^-=101$ and $y^-=100$, for $m=1$, $(x^-)_{m-j}=1\neq 0=y_{m-j}$ and $(y^-)_{m-i}=0=x_{m-i}$. Hence, by Lemma \ref{lem:interchangeable}, $u$ and $v$ are not interchangeable.

Note that it is possible for a pair of vertices to be interchangeable by two different automorphisms. For example,  $z = (0, 000)$ and $v=(5,000)$ in $\px(10,3)$ satisfy both conditions (2) and (3), and so can be interchanged using either a rotation or a reflection.

We will find it useful in the next section to identify which vertices of $\px(n,k)$ are interchangeable with $z=(0, 00\cdots 0)$.

\begin{cor}\label{cor:zInterchangeable}
Let $v=(j,y)\in V(\px(n,k))$ and let $$M = \big\{0,1,2,\dots,k-1\big\}\cap\big\{j,j+1,j+2,\dots,j+k-1\big\}.$$
Then $z=(0,00\cdots0)$ and $v$ are interchangeable by some $\alpha \in \mathcal A$
if and only if one of the following holds: 
\begin{enumerate}[(1)]
    \item $j=0$,
    \item $j \neq 0$ and for all $m\in M$, $y_{m-j}=(y^-)_m= y_{k-1-m}$, 
    \item $j=\frac{n}{2}$ and for all $m\in M$, $y_{m-j}= y_m$.
\end{enumerate}
\end{cor}

To illustrate how vertex interchangeability can be used to compute symmetry parameters, we consider the smallest twin-free Praeger-Xu graph.

\begin{lem}\label{lem:3,2interchange}
Any two distinct vertices of $\px(3,2)$ are interchangeable.
\end{lem}

\begin{proof}
    By vertex-transitivity, it suffices to show that $z = (0,00)$ is interchangeable with any vertex $v = (j,y)$. Since $n = 3$ is odd, we need only check that condition (2) of  Corollary~\ref{cor:zInterchangeable} holds. If $j = 1$, then $M=\{1\}$ and $y_{m-j} = y_{k-1-m} = y_0$. If $j = 2$, then $M =\{0\}$ and $y_{m-j} = y_{k-1-m} = y_1$.
\end{proof}

\begin{thm}\label{thm:dist3,2}
$\dist(\px(3,2))=2$ and $\rho(\px(3,2))=3$.
\end{thm}
\begin{proof}
Color the vertices in $R=\{(0,00), (1,01), (2,00)\}$ red and every other vertex blue. Assume $\alpha=\tau\delta\in\mathcal{A}=\aut(\px(3,2))$ preserves colors, where  $\tau=\tau_0^{u_0}\tau_1^{u_1}\tau_2^{u_2}$.

Suppose $\delta=\rho^s$ for some $s\in\{1,2\}$. If $s=1$, then  $\delta\cdot(1,01) =(2,01)$.
Since $\alpha$ preserves colors, $\tau\cdot (2,01)=(2,00)$,  which implies $u_0=1$ and $u_2 = 0$. However, then  $\tau\delta\cdot(2,00)= \tau\cdot(0,00)=(0,1a)$ for some $a\in\Z_2$. This contradicts our assumption that $\alpha$ preserves colors. If $s=2$, then $\delta\cdot(1,01)=(0,01)$, and hence $u_1=1$  and $u_0=0$. A similar contradiction arises because $\delta\cdot(2,00)=(1,00)$, and $u_1=1$ means $\tau\cdot(1,00)=(1,1a)$ for some $a\in\Z_2$. Thus $\delta$ cannot be a nontrivial rotation.

Suppose $\delta=\mu_s$ for some $s \in \mathbb Z_3$. By Lemma~\ref{lem:ijReflec}, $\delta$ preserves one fibre and interchanges the other two.
If $\delta$ preserves $\F_0$ and interchanges $\F_1$ and $\F_2$, then $\delta\cdot (1,01)=(2,10)$. Since $\alpha$ preserves colors, $\tau\cdot(2,10)=(2,00)$, so $u_2=1$. However, $\delta\cdot(2,00)=(1,00)$, which implies $\tau\cdot(1,00)=(1,1a)$ for some $a\in\Z_2$, a contradiction. 
Similar arguments apply to the remaining two cases.
Since $\delta$ is neither a nontrivial rotation nor a reflection, $\delta = \text{id}$.
 
Since $\alpha$ preserves colors and fibres, $\tau$ preserves the one red vertex in each fibre, so by Lemma~\ref{lem:tauPreservesEven},
$u_0=u_1=u_2 = 0$. Hence, $\alpha=\tau=\text{id}$, so this is a $2$-distinguishing coloring with smaller color class of size $3$. Thus, $\dist(\px(3,2))=2$ and $\rho(\px(3,2))\leq3$. 

To show $\rho(\px(3,2)> 2$, let $R=\{u,v\}\subset V(\px(3,2))$. Color the vertices in $R$ red and every other vertex blue. By Lemma~\ref{lem:3,2interchange}, 
$u$ and $v$ are interchangeable; any automorphism interchanging them is a nontrivial color-preserving automorphism.
Thus, $\rho(\px(3,2))> 2$, so $\rho(\px(3,2))=3$.
\end{proof}

\section{Distinguishing $\px(n,k)$, $k \ge 2$}\label{sec:DistNoTwins}

We have already found the distinguishing parameters for a number of Praeger-Xu graphs. Theorem~\ref{thm:DetkIsOne} covers the case $k=1$; Theorem~\ref{thm:dist3,2} covers $\px(3,2)$ and Proposition~\ref{prop:PX42} covers $\px(4,2)$. The next result covers the exceptional case $\px(4,3)$. The remainder of this section covers the case $n \ge 5$ and $k \ge 2$.

\begin{thm}\label{thm:dist43}
 $\dist(\px(4,3)) = 2$ and $\rho(\px(4,3))= 3 = \lceil \frac{4}{3}\rceil +1$.   
\end{thm}

\begin{proof} 
Color the vertices in $R = \{(0,000), (2,000), (3,001)\}$ red and all other vertices blue. 
Suppose $\beta \in \aut(\px(4,3))$ preserves this coloring. 
Recall that $\aut(\px(4,3))$ can be partitioned into the cosets $\mathcal A$ and $\mathcal A \xi$. First
assume $\beta = \alpha \xi$ for some $\alpha \in \mathcal A$. Then by assumption, 
\[
 R = \beta\big ( \{ (0,000),  (2,000),  (3,001)\}\big )
= \alpha(\{ (0,000), (0,001), (1,010)\}).
\]
Note that $R$ contains vertices in three different fibres, but 
since the fibres form a block system for any   $\alpha \in \mathcal A$,  
$\alpha(\{(0,000), (0,001), (1,010)\})$ contains two vertices in one fibre and a third vertex in a different fibre. So these two sets cannot be equal. Hence $\beta \notin \mathcal A \xi$.

Thus $\beta \in \mathcal A$, so $\beta = \tau \delta$ for some $\delta \in \langle \rho, \mu\rangle$.
Note that $(2,000)$ and  $(3,001)$ are adjacent, but neither is adjacent to $(0,000)$. Thus $\beta$ fixes $(0,000)$. If the induced action of $\beta$ on the fibres fixes $\F_0$, then either $\delta =$ id or $\delta = \mu$.  Since $\mu$  does not interchange fibres $\F_2$ and $\F_3$, $\delta = $ id. Thus $\beta$ fixes every vertex in $R$, which contains $ \{(0,000), (3,001)\}$, the determining set for $\px(4,3)$ found in Proposition~\ref{prop:detFourThree}. Hence $\beta = \text{id}$.
Thus this is  a 2-distinguishing coloring, proving that $\dist(\px(4,3)) = 2$.

Next we show that we cannot create a 2-distinguishing coloring with  fewer red vertices. If $R = \{(i, x)\}$, then $\tau_{i-1}$ is a nontrivial automorphism preserving the coloring. To show that no two-element set of red vertices provides a distinguishing coloring, it suffices, by vertex transitivity, to show that every vertex in $\px(4.3)$ is interchangeable with $z=(0,000)$.
Corollary~\ref{cor:zInterchangeable} shows that
 $z=(0,000)$ is interchangeable with every vertex in $\px(4,3)$ by some $\alpha \in \mathcal A$ except those listed below:
\begin{equation}
(1,010), (1,011), (1,100), (1,101), (3,010),(3,110),(3,001), (3,101).
\tag{*}
\end{equation}
For each vertex (*), we can find $\alpha \in \mathcal A$ such that $\alpha \xi$ interchanges it with $(0,000)$. For $(1,010)$, we seek $\alpha \in \mathcal A$ that satisfies
$\alpha (\xi \cdot (0,000)) = (1,010)$ and $\alpha (\xi \cdot (1,010)) = (0,000)$.
Referring to Table~\ref{tab:2CyclesInXi} for the action of $\xi$, we seek $\alpha \in \mathcal A$ such that 
$\alpha \cdot (0,000) = (1,010)$  and $\alpha \cdot (3,001)) = (0,000)$.
It is easy to verify that $\alpha = \tau_2 \rho$ satisfies this condition. For each vertex $v$ in (*), Table~\ref{tab:interchangingAlphas} gives an $\alpha$ satisfying $\alpha\xi \cdot (0,000) =\alpha\cdot (0,000)= v$ and $\alpha\xi\cdot v = (0,000)$. 
\end{proof}

\renewcommand{\arraystretch}{1.2}
\begin{table}[h]
     \centering
     \begin{tabular}{|c|c|c||c|c|c|}
     \hline
      $v$ & $\xi \cdot v$ & $\alpha$  & $v$ & $\xi \cdot v$ & $\alpha$ \\
      \hline
      $(1,010)$ & $(3,001)$ & $\tau_2 \rho$ &   $(3,010)$ & $(3,101)$ & $\tau_0 \tau_2 \mu_1$\\
      \hline
       $(1,011)$ & $(1,100)$ & $\tau_2 \tau_3 \mu_3$  & $(3,110)$ & $(1,101)$ & $\tau_0 \tau_2 \tau_3 \rho^3$\\
      \hline
       $(1,100)$ & $(1,011)$ & $\tau_0 \tau_1 \mu_3$  &  $(3,001)$ & $(1,010)$ & $\tau_1 \rho^3$\\
      \hline
       $(1,101)$ & $(3,110)$ & $\tau_0 \tau_1  \tau_3 \rho$ &   $(3,101)$ & $(3,010)$ & $\tau_1 \tau_3 \mu_1$\\
      \hline
     \end{tabular}
     \caption{$\alpha \in \mathcal A$ such that $\alpha\xi$ interchanges $(0,000)$ and $v$ in $(*)$.}
      \label{tab:interchangingAlphas}
     \end{table}

\begin{thm}
\label{thm:2DistCostBounds}
Let $n\geq5$ and $k \ge 2$. Then $\dist(\px(n,k))=2$ and\\ $\lceil \frac{n}{k}\rceil \le \rho(\px(n,k))\leq\lceil\frac{n}{k}\rceil+1$. 
\end{thm}

\begin{proof}
Let $x = 00\cdots 0, y = 11\cdots 1 \in \mathbb Z_2^k$. 
Then let \[
 R = \big\{(ik, x) : i \in \{0, 1, \dots, \lceil\tfrac{n}{k} \rceil -1\}\big\} \cup \big\{(1, y)\big\}.
\]
Color the vertices in $R$ red and all other vertices blue. Assume $\alpha=\tau\delta
\in\aut(\px(n,k))$ preserves these color classes. 
Then the induced action of $\delta$ on the fibres must preserve the set $I=\{0, 1, k, 2k, \dots, (\lceil\tfrac{n}{k} \rceil -1)k\} \subset \mathbb Z_n$. 
Note that $|R|= |I| = \lceil\tfrac{n}{k} \rceil +1 < n$. Interpreting $\mathbb Z_n$ as the vertex set of the cycle $C_n$, the (non-spanning) subgraph of $C_n$ induced by $I$ consists of a path containing at least the vertices $0$ and $1$, and possibly some isolated vertices. Let $F\subset \mathbb Z_n$ denote the set of vertices in the path; these will be the indices corresponding to a set of adjacent fibres of $\px(n,k)$ containing red vertices. Note that the action of $\delta$ on $C_n$ must preserve $F$. Since no nontrivial rotation preserves a proper subpath of $C_n$, $\delta \neq \rho^s$ for any $0 \neq s\in\Z_n$. So assume $\delta=\mu_s$ for some $s\in\Z_n$, and as usual, $\tau= \tau_0^{u_o} \cdots \tau_{n-1}^{u_{n-1}}$.

First assume $k=2$. If $n \ge 5$ is odd, then  $I= \{0, 1, 2, 4, \dots, n-1\}$ and $F=\{n-1,0,1,2\}$. For $F$ to be preserved under reflection, $\mu_s$ must interchange the vertex pairs $\{0, 1\}$ and $\{n-1, 2\}$ in $C_n$. By Lemma~\ref{lem:ijReflec}, $s  = 2$. 
In $\px(n,2)$, $\alpha = \tau \mu_2$ must interchange the vertex pairs $\{(0,00),(1,11)\}$ and $\{(n-1,00), (2,00)\}$.
Thus $\tau\cdot(0,11)=(0,00)$ and $\tau\cdot(1,11)=(1,00)$, which implies
 $u_0=u_1=u_2=1$. However, it must also be the case that $\tau\cdot(n-1,00)=(n-1,00)$, so $u_0=0$, a contradiction.

If instead $n\ge 5$ is even, then $I = \{0, 1, 2, 4, \dots, n-2\}$ and $F=\{0,1,2\}$. In this case, $\mu_s$ must fix $1$ and interchange $0$ and $2$, so by Lemma~\ref{lem:ijReflec}, $s=3$. In this case, as an element of $\aut(\px(n,2))$, $\mu_3$ is a nontrivial automorphism preserving $R$, so this does not define a $2$-distinguishing coloring. However, let \[
 R' = \{(0, y)\} \cup \big\{(ik, x) : i \in \{1, \dots, \lceil\tfrac{n}{k} \rceil -1\}\big\} \cup \big\{(1, x)\big\}.
\] 
Then $I'=I = \{0, 1, 2, 4, \dots, n-2\}$ and $F'=F=\{0, 1, 2\}$. The only reflection preserving $F'$ is still $\mu_3$. 
If $\alpha = \tau\mu_3$ preserves $R'$, 
then $\tau\mu_3\cdot (0,11) = \tau\cdot(2,11) = (2, 00)$, which means $u_2 = u_3 = 1$. 
Also, $\tau\mu_3\cdot (2,00) = \tau\cdot(0,00) = (0, 11)$, so $u_0=u_1 = 1$. 
This creates a contradiction because $\tau 
\mu_3\cdot(1,00) = \tau(1,00) = (1,00)$, 
which implies $ u_1 = u_2 = 0$.

Now assume $k>2$. If $n\neq 1\bmod k$, then $ F= \{0,1\}$. Then $\mu_s$ must interchange $0$ and $1$, so by Lemma~\ref{lem:ijReflec}, $s = k$.  Since $\alpha = \tau \mu_k$ preserves $R$, $\tau\cdot(0,y)=(0,x)$ and $\tau\cdot(1,x)=(1,y)$. 
Then $u_0=u_1=\cdots=u_k=1$. Since $k \in I$ and $\mu_k$ preserves $I$, $\mu_k(j) = k$ for some $j \in I\setminus \{0, 1\}$. Then $\alpha(j,x) = \tau\mu_s(j,x) = \tau\cdot (k,x)=(k,x)$, so $u_k=0$, a contradiction.

If instead $n=1\bmod k$, then $F=\{n-1,0,1\}$. 
Then $\mu_s$ fixes $0$ and interchanges $n-1$ and $1$, so $s=k-1$. Then $\tau\cdot(n-1,y)=(n-1,x)$ and $\tau\cdot(1,x)=(1,y)$. Hence $u_{n-1}=u_0=\cdots=u_k=1$. However, $\tau\cdot(0,x)=(0,x)$, so $u_0=u_1=\cdots=u_{k-1}=0$, a contradiction.

Thus $\delta\neq \mu_s$ for any $s \in \mathbb Z_n$, so $\delta=\text{id}$ and hence $\alpha=\tau \in K$.
For every $0 \le t \le \lceil{\frac{n}k}\rceil - 1$,
$\F_{tk}$ contains exactly one red vertex that is fixed by $\tau$, so by Lemma~\ref{lem:tauPreservesEven}, $u_{tk} = u_{tk+1} = \dots = u_{tk+k-1} = 0$. Hence $u_0 = \dots = u_{n-1} = 0$ and so $\tau = \text{id}$.
Thus, this is a $2$-distinguishing coloring of $\px(n,k)$ with a color class of size $\lceil\frac{n}{k}\rceil+1$, so $\dist(\px(n,k))=2$ and $\rho(\px(n,k))\leq \lceil\tfrac{n}{k}\rceil+1$.

To establish the lower bound on cost,
assume there exists a set of vertices $R=\{u_1,u_2,\dots,u_r\}$ with $r<\lceil\frac{n}{k}\rceil$ such that coloring the vertices of $R$ red and all other vertices blue defines a $2$-distinguishing coloring of $\px(n,k)$. 
If $\alpha\in\aut(\px(n,k))$ fixes every vertex in $R$, then certainly $\alpha$ preserves the color classes and so by assumption $\alpha = \text{id}$.
Hence, $R$ is a determining set of size $r<\lceil\frac{n}{k}\rceil$, a contradiction of Theorem~\ref{thm:DET}. Thus, $\rho(\px(n,k))\geq \lceil\tfrac{n}{k}\rceil$.
\end{proof}

\begin{figure}[h]
\includegraphics[scale=0.3, center]{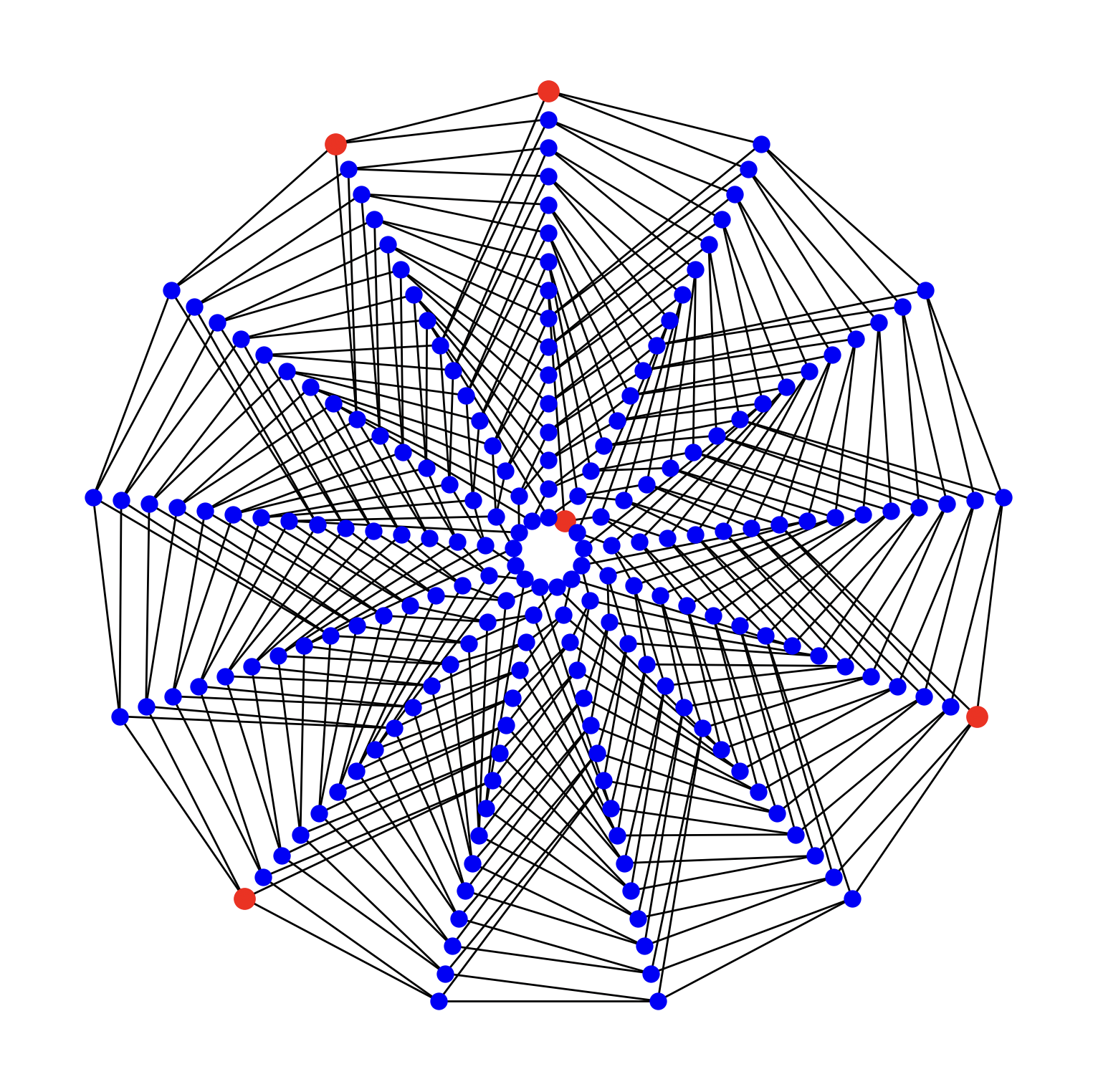}
\caption{$\px(13,4)$ with a $2$-distinguishing coloring of cost $5=\lceil\tfrac{13}{4}\rceil +1$.}
\label{fig:PX(13, 4)}
\end{figure}

\medskip

The remaining theorems indicate which Praeger-Xu graphs (for $n\ge 5$ and $k \ge 2$) have cost  $\lceil\frac{n}{k}\rceil$ and which have cost $\lceil\frac{n}{k}\rceil+1$.

\begin{thm}
Let $n \ge 5$ and $2\leq k<n$. 
If $k$ divides $n$, then $\rho(\px(n,k))=\lceil\frac{n}{k}\rceil+1=\frac{n}{k}+1$.
\end{thm}
\begin{proof}
Let $R\subset V$ be any set of $\frac{n}{k}$ vertices. Color every  vertex in $R$ red and every other vertex blue. 
It suffices to show we can always find a nontrivial automorphism preserving $R$.

Let $I=\{i_1,i_2,\dots,i_r\} \subseteq \mathbb Z_n$ be the set of indices of fibres containing red vertices, where we assume that as integers, $0\le i_1 < i_2 < \dots < i_r <n$.  Then the gaps between these fibres 
contain $i_2-i_1-1,i_3-i_2-1,\dots, n+i_1-i_r-1$ fibres, respectively. If there exists $i_p\in I$ such that the gap between $i_p$ and $i_{p+1}$ contains at least $k$ fibres, then $\tau_{i_{p+1}-1}$ is a nontrivial automorphism that preserves colors. So assume that for all $i_p\in I$, the gap between $i_p$ and $i_{p+1}$ contains fewer than $k$ fibres.

Suppose there exists $i_p\in I$ such that the gap between $i_p$ and $i_{p+1}$ contains fewer than $k-1$ fibres. Since $r=|I|\leq |R|=\frac{n}{k}$, the total number of fibres is strictly less than $r+(k-1)r = kr \le k \cdot \frac{n}{k} = n$, a contradiction.
Thus for all $i\in\Z_n$, $i\in I$ if and only if $i+k\in I$, so the induced action of $\rho^k$ preserves $I$ as a subset of $V(C_n)$.

Since the fibres containing red vertices are separated by  $k-1$ fibres, 
we can define $\tau\in K$ such that $\tau$ adjusts the bitstring components of vertices in these fibres independently. More precisely, let $\tau=\tau_0^{u_0}\tau_1^{u_1}\cdots\tau_{n-1}^{u_{n-1}}$, where $u_m=1$ if and only if there exist $(i_p,x), (i_{p+1},y)\in R$ such that  $x_{m-i_p}\neq y_{m-i_{p+1}}$.
Then for all $(i_p,x)\in R$, 
$\tau\rho^k\cdot(i_p,x))=\tau\cdot(i_{p+1},x)=(i_{p+1},y)\in R.$
Thus, $\tau\rho^k$ is a nontrivial automorphism that preserves colors.
\end{proof}

\begin{thm} If $5 \le n < 2k$, then $\rho(\px(n,k)) = \lceil \frac{n}{k} \rceil = 2$ .
\end{thm}

\begin{proof}
Let $j =  \lfloor \tfrac{n}{2} \rfloor - 1$. Then  $5 \le n < 2k$ implies $0 < j < k-1$. 
Next, let $R = \{z, v\}$ where $z = (0, 000\cdots 0)$ and $v = (j,y) = (j, 011\cdots 1))$.
 Color vertices in $R$ red and all other vertices blue; assume $\alpha \in \aut(\px(n,k))$ preserves these color classes. 
 Let $M = \big\{0,1,\dots,k-1\big\}\cap\big\{j,j+1,\dots,j+k-1\big\}.$
 Since $0< j < k-1$,  $j \in M$.
 For $m = j$, $y_{m-j} = y_0 = 0$, but $y_{k-1-m} = y_{k-1-j} = 1.$ By Corollary~\ref{cor:zInterchangeable}, $z$ and $v$ are not interchangeable, so $\alpha$ can only preserve $R$ by fixing $z$ and $v$.
 Because fibres $\F_0$ and $\F_j$ are not antipodal, $R$ is a determining set by Theorem~\ref{thm:DetnNot4}.  By definition, $\alpha$ is the identity. Thus we have defined a $2$-distinguishing coloring in which the smaller color class has size $2$.
 \end{proof}

\begin{thm}
Let $k\geq 2$ and $n>2k$ such that $k$ does not divide $n$. Then 
$$\rho(\px(n,k)) = 
\begin{cases} 
\lceil\tfrac{n}{k}\rceil+1, &\text{if } 
n = -1 \bmod k,\\
\lceil\tfrac{n}{k}\rceil, \quad &\text{if } 
n \neq -1 \bmod k. 
\end{cases}$$
\end{thm}

\begin{proof}
First assume $n= -1\bmod k$, so $n = \lceil \frac{n}{k} \rceil k - 1$. 
Let $R$ be any set of $\lceil\frac{n}{k}\rceil$ vertices. Color every vertex in $R$ red and every other vertex blue.
Let $I=\{i_1,i_2,\dots,i_r\} \subset \mathbb Z_n$ be the set of indices of the fibres containing red vertices, where as integers, $0\leq i_1 < i_2<\cdots  <i_r<n$. We will show that there is a nontrivial automorphism preserving $R$.

If there exists $i_p\in I$ such that the gap between $i_p$ and $i_{p+1}$ contains at least $k$ fibres, then $\tau_{i_p+k}$ is a nontrivial automorphism that preserves colors. So assume that every gap has at most $k-1$ fibres. Suppose there exist at least two gaps that contain at most $k-2$ fibres. 
Then the total number of fibres is at most 
\[
r+2(k-2)+(r-2)(k-1)= rk-2 \leq \lceil\tfrac{n}{k}\rceil k -2 <\lceil \tfrac{n}{k} \rceil k - 1 = n,
\]
a contradiction. Thus at most one gap contains at most $k-2$ fibres and the others contain exactly $k-1$ fibres. 
If two vertices $u,v\in R$ are in the same fibre, then  $r<\lceil\frac{n}{k}\rceil$ and  then the total number of fibres is 
\[
r+(r-1)(k-1)+k-2 = rk-1< \lceil\tfrac{n}{k}\rceil k -1 =n,
\]
a contradiction. 
Thus  $ r=\lceil\frac{n}{k}\rceil$, every gap 
except one contains $k-1$ fibres, and the remaining gap contains $k-2$ fibres.  
By vertex-transitivity, we can assume  $I=\{0,k,2k,\dots,(\lceil\tfrac{n}{k}\rceil-1)k\}$.

Let $j=(\lceil\tfrac{n}{k}\rceil-1)k= n-(k-1)$.  The gap between $\F_j$ and $\F_0$ is the one containing exactly $k-2$ fibres; all other gaps contain $k-1$ fibres. Let $u=(0,x),v=(j,y)\in R$ be the red vertices in $\F_0$ and $\F_j$, respectively.
Then, as defined in Lemma~\ref{lem:interchangeable}, let
\[
M =  \big\{0,1,\dots,k-1\big\}\cap\big\{j, j+1, \dots, j+k-1\}
= \big\{0\big\}.
\]
For the only $m\in M$, 
$m-j=k-1$
and 
$m-i=0.$
Then $(x^-)_{m-j}=x_0$, $y_{m-j}=y_{k-1}$, $(y^-)_{m-i}=(y^-)_{0}=y_{k-1}$, and $x_{m-i}=x_0$. Of course, $x_0=y_{k-1}$ if and only if $y_{k-1}=x_0$. By Lemma~\ref{lem:interchangeable} $u$ and $v$ are interchangeable by an automorphism  of the form by $\alpha=\tau\mu_s\in\mathcal{A}$, where 
$s=n=0 \bmod n$
and 
$\tau=\tau_0^{u_0}\tau_1^{u_1}\cdots\tau_{n-1}^{u_{n-1}}$ is designed to flip exactly the right bits of the bitstring components of vertices in $\F_0$ and $\F_j$.
More precisely, let $t \in \{0, 1, \dots, k-1\}$. If $y_t \neq (x^-)_t = x_{k-1-t}$, then $u_{j+t}=u_{k-1-t} = 1$; 
if $y_t = (x^-)_t = x_{k-1-t}$, then $u_{j+t}=u_{k-1-t} = 0$.
Note that this prescribes the value of $u_m$ for all $m \in \{0, 1, \dots k-1\} \cup \{j, j+1, \dots, j+k-1\}$; vertices in $\F_0$ and $\F_j$ are unaffected by the value $u_\ell$ for any $\ell\in \{k, k+1, \dots, j-1\}$.

We claim that we can set the value of $u_\ell$ for all $\ell\in \{k, k+1, \dots, j-1\}$ in such a way that $\tau\mu_0$ preserves $R$. Note that 
\[
\{k, k+1, \dots, j-1\} = \bigsqcup_{a=1}^ {\lceil\frac{n}{k}\rceil-2} \{ak, ak+1, ak+2, \dots, ak+k-1\}.
\]
Let $b \in \{1, \dots,\lceil\frac{n}{k}\rceil-2 \}$ and let $(bk, w)$ be the red vertex in $\F_{bk}$. 
Then $\tau\mu_n \cdot (bk, w) =  \tau\cdot (ak, w^-)$, for some $a \in \{1, \dots,\lceil\frac{n}{k}\rceil-2 \} \setminus \{b\}$.  
We can arrange to have $\tau(ak, w^-)$ equal the red vertex in $\F_{ak}$ by flipping bits in $w^-$ as necessary; this can be achieved by appropriately setting the values of $u_{ak}, u_{ak+1}, \dots , u_{ak+k-1}$. These values won't affect vertices in any of the other fibres containing red vertices.
 
Now assume $n\neq -1\bmod k$; because we are also assuming that $k \nmid n$, the remainder after $n$ is divided by $k$ satisfies  
$0<n-(\lceil\frac{n}{k}\rceil-1)k< k-1$.
Again, let $I=\{0,k,2k,\dots,(\lceil\frac{n}{k}\rceil-1)k\}$. To simplify notation, again let $j =(\lceil\frac{n}{k}\rceil-1)k$. Then let
\[
R = \{(i, 00\cdots 00) : i\in I \setminus \{j\}\} \cup \{(j, 00\cdots 01)\}.
\]
Note that $|R| = |I| = \lceil\frac{n}{k}\rceil$. Color every vertex in $R$ red and all other vertices blue.
Let $\alpha=\tau\delta\in\mathcal{A}=\aut(\px(n,k))$ such that $\alpha$ preserves these color classes.
The induced action of $\alpha$ on the fibres must preserve the set $I$, interpreted as a subset of $V(C_n)$. The distance between $0$ and $j=(\lceil\frac{n}{k}\rceil-1)k$ in $C_n$ is strictly less than $k+1$, whereas the distance between any other two consecutive elements of $I$ in $C_n$ is exactly $k+1$. So no nontrivial rotation preserves $I$.

Thus $\delta=\mu_s$, where $\mu_s$ interchanges $0$ and $j$. 
Then $\tau\mu_s$ interchanges the red vertices in $\F_0$ and $\F_j$, so $\tau\delta\cdot(0,00\cdots0)=\tau (j,00\cdots0) = (j,00\cdots01)$, which implies $u_j = u_{j+1} = \dots =u_{j+k-2} = 0$ and $u_{j+k-1} =1$. 
Additionally, $\tau\delta\cdot(j,00\cdots01)=\tau \cdot(0,10\cdots00) = (0, 00\cdots00)$, which implies $u_0=1$ and $u_2 = u_3 = \dots = u_{k-1} = 0$. 
A contradiction arises because $0<n-j< k-1$ implies that in $\mathbb Z_n$, $0 =n = j+m$ for some $m \in \{1, 2, \dots, k-2\}$.
Thus, $\delta=\text{id}$ and so $\alpha = \tau=\tau_0^{u_0}\tau_1^{u_1}\cdots\tau_{n-1}^{u_{n-1}} \in K$.

For every $0 \le t \le \lceil{\frac{n}k}\rceil - 1$,
$\F_{tk}$ contains exactly one red vertex that is fixed by $\tau$, so by Lemma~\ref{lem:tauPreservesEven}, $u_{tk} = u_{tk+1} = \dots = u_{tk+k-1} = 0$. Hence $u_0 = \dots = u_{n-1} = 0$ and so $\tau = \text{id}$. Thus, this is a $2$-distinguishing coloring of $\px(n,k)$ of with a color class of size $\lceil\frac{n}{k}\rceil$. By Theorem~\ref{thm:2DistCostBounds}, $\rho(\px(n,k))=\lceil\frac{n}{k}\rceil$.
\end{proof}

Our results on distinguishing number and cost are summarized below.

\begin{thm}
Let $n\geq3$ and $2\leq k<n$. Then $\dist(\px(n,k))=2$ and 
$$
\rho(\px(n,k))=
\begin{cases}
5, \, & \text{ if } (n,k) = (4,2),\\
\lceil\tfrac{n}{k}\rceil, &\text{ if } 5\leq n<2k \text{ or }\\
& \text{ if } 2k<n \text{ and } n \notin \{0\bmod k,-1\bmod{k}\},\\
\lceil\tfrac{n}{k}\rceil +1, &\text{ otherwise.}
\end{cases}
$$
\end{thm}

\bibliographystyle{abbrv}
\bibliography{PXBibliography}

\end{document}